\documentclass[letterpaper, 11pt,  reqno]{amsart}
\usepackage[margin=1.2in, marginparwidth=1.5cm, marginparsep=0.5cm]{geometry}

\usepackage{amsmath,amssymb,amscd,amsthm,amsxtra, esint}

\usepackage[implicit=true]{hyperref}

\usepackage{color}

\allowdisplaybreaks[2]

\definecolor{gr}{rgb}   {0.,   0.69,   0.23 }
\definecolor{bl}{rgb}   {0.,   0.5,   1. }
\definecolor{mg}{rgb}   {0.85,  0.,    0.85}
\definecolor{yl}{rgb}   {0.8,  0.7,   0.}
\definecolor{or}{rgb}  {0.7,0.2,0.2}

\newtheorem{theorem}{Theorem} [section]

\newtheorem{lemma}[theorem]{Lemma}

\newtheorem{remark}[theorem]{Remark}

\newtheorem{corollary}[theorem]{Corollary}

\newcommand{\noi}{\noindent}
\newcommand{\Z}{\mathbb{Z}}
\newcommand{\R}{\mathbb{R}}

\newcommand{\T}{\mathbb{T}}

\let\P= \undefined
\newcommand{\P}{\mathbf{P}}

\newcommand{\E}{\mathbb{E}}

\newcommand{\F}{\mathcal{F}}

\newcommand{\Nf}{\mathfrak{N}}

\newcommand{\al}{\alpha}
\newcommand{\be}{\beta}
\newcommand{\dl}{\delta}

\newcommand{\nb}{\nabla}

\newcommand{\Dl}{\Delta}
\newcommand{\eps}{\varepsilon}

\newcommand{\g}{\gamma}

\newcommand{\s}{\sigma}

\newcommand{\ft}{\widehat}

\newcommand{\wt}{\widetilde}
\newcommand{\cj}{\overline}

\newcommand{\dt}{\partial_t}

\renewcommand{\l}{\ell}
\renewcommand{\o}{\omega}

\newcommand{\les}{\lesssim}
\newcommand{\ges}{\gtrsim}

\newcommand{\jb}[1]
{\langle #1 \rangle}

\newcommand{\ind}{\mathbf 1}

\newcommand{\N}{\mathbb{N}}

\newcommand{\NN}{\mathcal{N}}

\renewcommand{\H}{\mathcal{H}}

\newtheorem*{ackno}{Acknowledgements}

\numberwithin{equation}{section}
\numberwithin{theorem}{section}

\usepackage{tikz} 
\usepackage{marginnote}
\usepackage{scalerel} 

\usetikzlibrary{shapes.misc}
\usetikzlibrary{shapes.symbols}
\usetikzlibrary{decorations}
\usetikzlibrary{decorations.markings}

\tikzset{
	dot/.style={circle,fill=black,draw=black,inner sep=0pt,minimum size=0.5mm},
	>=stealth,
	}

\tikzset{
	dot2/.style={circle,fill=black,draw=black,inner sep=0pt,minimum size=0.2mm},
	>=stealth,
	}

\tikzset{
	ddot/.style={circle,fill=white,draw=black,inner sep=0pt,minimum size=0.8mm},
	>=stealth,
	}

\tikzset{decision/.style={ 
        draw,
        diamond,
        aspect=1.5
    }}

\tikzset{dia2/.style
={diamond,fill=white,draw=black,inner sep=0pt,minimum size=1mm},
	>=stealth,
	}

\tikzset{dia/.style
={star,fill=black,draw=black,inner sep=0pt,minimum size=1mm},
	>=stealth,
	}

\tikzset{dia/.style
={diamond,fill=black,draw=black,inner sep=0pt,minimum size=1.3mm},
	>=stealth,
	}

\makeatletter

\def\<#1>{\xusebox{#1}}
\makeatother

\makeatletter
\@namedef{subjclassname@2020}{%
  \textup{2020} Mathematics Subject Classification}
\makeatother

\newcommand{\too}{\longrightarrow}

\begin{document}
\baselineskip = 13pt

\title[Randomization of a general function
of  negative regularity]
{
A remark on randomization of a general function
of  negative regularity
} 
\author[T.~Oh , M.~Okamoto, O.~Pocovnicu, and   N.~Tzvetkov]
{Tadahiro Oh, Mamoru Okamoto, Oana Pocovnicu,  and Nikolay Tzvetkov}

\address{
Tadahiro Oh, School of Mathematics\\
The University of Edinburgh\\
and The Maxwell Institute for the Mathematical Sciences\\
James Clerk Maxwell Building\\
The King's Buildings\\
Peter Guthrie Tait Road\\
Edinburgh\\ 
EH9 3FD\\
 United Kingdom}

\email{hiro.oh@ed.ac.uk}

\address{
Mamoru Okamoto\\
Department of Mathematics\\
 Graduate School of Science\\ Osaka University\\
Toyonaka\\ Osaka\\ 560-0043\\ Japan}
\email{okamoto@math.sci.osaka-u.ac.jp}

\address{
Oana Pocovnicu\\
Department of Mathematics, Heriot-Watt University and The Maxwell Institute for the Mathematical Sciences, Edinburgh, EH14 4AS, United Kingdom}
\email{o.pocovnicu@hw.ac.uk}

\address{
Nikolay Tzvetkov\\
Ecole Normale Sup\'erieure de Lyon, UMPA, UMR CNRS-ENSL 5669, 46, all\'ee d'Italie\\
 69364-Lyon Cedex 07, France}

\email{nikolay.tzvetkov@ens-lyon.fr}

\subjclass[2020]{35R60, 35L05, 60H15}
\keywords{probabilistic well-posedness; Wick power;  random initial data; stochastic forcing; 
Fourier-Lebesgue space}

\begin{abstract}

In the study of  partial differential equations (PDEs) with random initial data 
and singular stochastic PDEs with random forcing, 
we typically decompose a classically ill-defined solution map
into two steps, 
where, in the first step, we use stochastic analysis to construct 
various stochastic objects.
The simplest kind of such stochastic objects
is the Wick powers of a basic stochastic term
(namely a random linear solution, a stochastic convolution, 
or their sum).
In the case of randomized initial data of a general function of negative regularity
for studying nonlinear wave equations (NLW), 
we show
necessity of  imposing additional Fourier-Lebesgue regularity
for constructing Wick powers
by 
exhibiting  examples of functions slightly outside $L^2(\T^d)$
such that the associated  Wick powers do not exist.
This shows that probabilistic well-posedness theory for NLW
with general randomized initial data  fails in negative Sobolev spaces
(even with renormalization).
Similar examples also apply to stochastic NLW
and stochastic nonlinear heat equations
with 
general white-in-time stochastic forcing, 
showing  necessity of appropriate Fourier-Lebesgue $\g$-radonifying regularity
in the construction of the Wick powers of 
the associated stochastic convolution.

\end{abstract}


%
\maketitle
\tableofcontents


\section{Introduction}

Over the last decade, there has been a significant progress
in the study of random dispersive PDEs, 
broadly interpreted with random initial data and/or stochastic forcing.
This study was initiated by Bourgain \cite{BO94, BO96}
in the construction of invariant Gibbs measures for 
the nonlinear Schr\"odinger equations (NLS)
and was further developed by Burq and the fourth author
\cite{BT1, BT3}
in the context of the nonlinear wave equations (NLW); see also \cite{CO}.
See~\cite{BOP4} for a  survey on the subject.
In recent years,
we have witnessed a rapid progress 
\cite{OTh2, GKO, OTW,  OPTz, GKOT, OOR, OOcomp, ORSW, ORSW2, 
ORW,  OWZ, STX, LTW}
in
probabilistic well-posedness theory for nonlinear dispersive PDEs
in the {\it singular} setting,\footnote{Namely, when a solution at a fixed time is not a function but only a distribution.}
culminated
in 
 the paracontrolled approach to NLW \cite{GKO2, OOT1, Bring, OOT2, BDNY}
and the introduction of 
random averaging operators and random tensors for NLS \cite{DNY2, DNY3}.

We point out that, 
in the singular setting, 
with the exception of the stochastic KdV equation studied in \cite{DDT, Oh4}
(see also \cite{FOW}), 
all the known probabilistic well-posedness results 
on nonlinear dispersive PDEs (including those with stochastic forcing)
are essentially limited to

\begin{itemize}
\item Gaussian free field initial data, white noise initial data, 
or their smoothed (or differentiated) versions in the case of random initial data;
see, for example, \eqref{rand3} below.

\smallskip

\item space-time white noise 
or its smoothed (or differentiated) version in the case of stochastic forcing;\footnote{We restrict our discussion to the white-in-time case.}
see, for example, \cite{OOcomp, OWZ}.

\end{itemize}

\noi
In this paper, we investigate 
issues related to 
general randomized initial data
 and general stochastic forcing
in the singular setting.
In particular, 
we 
 consider the deterministic NLW
on the  periodic torus $\T^d = (\R/2\pi \Z)^d$:\footnote{The equation \eqref{NLW0}
is also referred to as the nonlinear Klein-Gordon equation.
We, however, simply refer to \eqref{NLW0} as NLW in the following.
Moreover, we only consider real-valued functions in the following.
For a renormalization in  the complex-valued case, see 
 \cite{OTh1}.}
\begin{equation}
\partial_t^2 u +(1-\Delta) u+u^k =0
\label{NLW0}
\end{equation}

\noi
with random initial data, 
and the stochastic NLW (SNLW) on $\T^d$:
\begin{equation}
\partial_t^2 u +(1-\Delta) u+u^k = \Phi \xi, 
\label{NLW0a}
\end{equation}

\noi
where $k \ge 2$ is an integer
and the unknown function $u$ is real-valued.
In \eqref{NLW0a}, 
$\xi$ denotes the (Gaussian) space-time white noise
whose space-time covariance is formally given by 
\[ \E\big[\xi(t_1, x_1)\xi(t_2, x_2) \big] = \dl(t_1 - t_2) \dl(x_1 - x_2),\]

\noi
and $\Phi$ is a linear operator on $L^2(\T^d)$.
For conciseness of the presentation, 
we will only discuss details in the case of random initial data
in the following.
Analogous results hold in the case of general stochastic forcing;
see Subsection \ref{SUBSEC:sto}.

\subsection{Randomization of a general function}

In  \cite{BT1, BT3}, 
Burq and the fourth author studied
well-posedness of NLW \eqref{NLW0}
with randomization of general functions as initial data.
In the current setting of $\T^d$, 
given a pair of deterministic functions\footnote{By convention, 
 we endow $\T^d$ with the normalized Lebesgue measure $dx_{\T^d}= (2\pi)^{-d} dx$.}
\begin{equation}
u_0 = \sum_{n \in \Z^d} a_n e^{in\cdot x}
\qquad\text{and}\qquad 
u_1 = \sum_{n \in \Z^d} b_n e^{in\cdot x}
\label{rand1}
\end{equation}

\noi
with the constraint $a_{-n} = \cj{a_n}$ and $b_{-n} = \cj{b_n}$, $n \in \Z^d$, 
we consider the randomized initial data $(u_0^\o, u_1^\o)$ given by 
\begin{equation}
u_0^\o = \sum_{n \in \Z^d} g_n(\o)a_n e^{in\cdot x}
\qquad\text{and}\qquad 
u_1^\o = \sum_{n \in \Z^d}h_n(\o)b_n e^{in\cdot x}, 
\label{rand2}
\end{equation}

\noi
where the series   $\{ g_n \}_{n \in \Z^d}$ and  $\{ h_n \}_{n \in \Z^d}$ are two families of 
independent standard   complex-valued  Gaussian random variables  
conditioned that
$g_{-n}=\cj {g_{n}}$
and   $h_{-n}=\cj {h_{n}}$, $n \in \Z^d$.
In~\cite{BT1, BT3}, 
the authors considered
a more general class of random variables
 $\{ g_n,  h_n \}_{n \in \Z^d}$, 
 satisfying some (exponential) moment bound.
In the following, however, 
we restrict our attention to the Gaussian case.

Given $s \in \R$, let $\H^s(\T^d) = H^s(\T^d)\times H^{s-1}(\T^d)$, 
where 
$H^s(\T^d)$ denotes the standard $L^2$-based Sobolev space on $\T^d$, 
 endowed with the norm:
\[
\|f\|_{H^s(\T^d)}=\|\jb{n}^s\ft{f}(n)\|_{\l^2(\Z^d)}, \qquad \jb{\,\cdot\,} = (1+|\cdot|^2)^\frac{1}{2}.
\]

\noi
It is well known that if $(u_0, u_1) \in \H^s(\T^d)$, 
then 
 the randomized pair $(u_0^\o, u_1^\o)$
 is 
almost surely in $\H^s(\T^d)$.
Moreover, if $(u_0, u_1) \notin \H^{s+\eps}(\T^d)$
for some $\eps > 0$,
then
 $(u_0^\o, u_1^\o)
  \notin \H^{s+\eps}(\T^d)$ almost surely; see \cite[Lemma B.1]{BT1}.
While  there is no smoothing upon randomization in terms of differentiability in general, 
this randomization provides better integrability;
if $u_0  \in H^s(\T^d)$,
then  the randomized function $u_0^\o$ 
almost surely 
belongs to   $W^{s,p}(\T^d)$ for any finite $p \geq 1$.
This gain of integrability plays a crucial role
in proving probabilistic well-posedness of \eqref{NLW0}
for randomized initial data of supercritical  but {\it non-negative}\footnote{Here, we consider the regularity for $u_0$.} regularity, 
where the Cauchy problem is
known to be (deterministically) ill-posed \cite{CCT,  OOTz, FOk}.
See \cite{BT1, BT3, BTTz, OP2} for probabilistic well-posedness results on $\T^d$.
See also 
\cite{LM, Poc, OP1} for analogous results on $\R^d$.

Next, we consider  the case of negative regularity.
In this case, the known probabilistic well-posedness results
  \cite{OTh2, OPTz, OOTz, OZ, LP}
are limited to the random initial data of the following form:
\begin{equation}
\varphi_0^\o = \sum_{n \in \Z^d} \frac{g_n(\o)}{\jb{n}^{1+\al}}e^{in\cdot x}
\qquad\text{and}\qquad 
\varphi_1^\o = \sum_{n \in \Z^d} \frac{h_n(\o)}{\jb{n}^{\al}}e^{in\cdot x}.
\label{rand3}
\end{equation}

\noi
When $\al = 0$, $\varphi_0^\o$ corresponds to the massive Gaussian free field on $\T^d$, 
while $\varphi_1^\o$ corresponds to the (Gaussian) white noise on $\T^d$.
See \cite{BO96, CO, OTW, DNY2, DNY3}
in the case of NLS.
It is easy to see that  $(\varphi_0^\o, \varphi_1^\o)
\in \H^\s(\T^d)\setminus \H^{\al +1 - \frac d2}(\T^d)$
for any 
\begin{align}
\s < s(d, \al) \stackrel{\text{def}}{=} \al +1 - \frac d2
\label{al0}
\end{align}

\noi
and thus we restrict our attention to $\al \le \frac d 2- 1$ (such that $\s < 0$).
In this case, the random linear solution $Z(t)$ defined by 
\begin{align}
Z(t) =  \cos(t \jb{\nb}) \varphi_0^\o + \frac{\sin(t\jb{\nb})}{\jb{\nb}} \varphi_1^\o
\label{rand4}
\end{align}

\noi
is merely a distribution (for each $t \in \R$).
Indeed, 
by letting $Z_N = \P_N Z$, 
where $\P_N$ denotes the frequency projection onto frequencies $\{|n|\le N\}$, 
it follows from $\al \le \frac d 2- 1$ that\footnote{Due to the translation-invariance 
of the law of $Z(t, x)$, $\s_N$ is independent of $t \in \R$ and $x \in \T^d$.} 
\begin{align}
\s_N \stackrel{\text{def}}{=} \E\big[Z_N^2(t, x)\big]
= \sum_{\substack{n \in \Z^d\\|n|\le N}} \frac {\ind_{|n| \le N}}{\jb{n}^{2(1+\al)}}\too \infty, 
\label{sig1}
\end{align}

\noi
as $N \to \infty$.
As a result, 
we expect that a solution $u(t)$ is a distribution (for fixed $t$)
and thus the nonlinearity $u^k(t)$ in \eqref{NLW0} does not make sense, 
which necessitates us to renormalize the nonlinearity.
See the introductions in~\cite{OTh2, GKO, OOTz}.

In the following, for simplicity, we assume that 
 $\al$ is sufficiently close to $ \frac d2 - 1$.
Given $j \in \N$, 
define the truncated Wick  power:
\begin{equation}
:\!Z_N^j (t, x)\!:\, \,  =  H_j (Z_{N}(t, x); \s_N), 
\label{Wick1}
\end{equation}

\noi
where $H_j$ is the Hermite polynomial of degree $j$
and $\s_N$ is as in \eqref{sig1}.
Then, arguing as in \cite{GKO2, OPTz, OOTz}, 
the truncated Wick  power
$:\!Z_N^j \!:$ converges to 
a limit, denoted by  $:\!Z^j\!:$, 
in $C([0, T]; W^{j s(\al, d) - \eps, \infty}(\T^d))$
for any $\eps > 0$,  almost surely, 
as $N \to \infty$, 
where $s(d, \al)$ is as in~\eqref{al0}.
Then, 
the basic strategy to study probabilistic local well-posedness of 
(the renormalized version of)~\eqref{NLW0}, 
at least when 
 $\al$ is sufficiently close to $ \frac d2 - 1$, 
 is to write 
a solution $u$ in the first order expansion $u = Z + v$
and study the equation satisfied by $v$:
\begin{align}
\dt^2 v +(1-\Delta) v+\NN_k(v+Z) = 0
\label{NLW1}
\end{align}

\noi
with the zero initial data.
Here, $\NN_k$ denotes the Wick renormalized nonlinearity given by
\begin{align}
\NN_k (v + Z) =  \sum_{j = 0}^k {k \choose j} v^{k-j}   :\!Z^j\!:.
\label{Wick2}
\end{align}

\noi
We point out that the main task 
in this argument is the construction of the Wick powers $:\!Z^j\!:$.
Once this is achieved, local well-posedness of \eqref{NLW1}
for $v$
follows from a standard contraction argument via
Sobolev's inequality and/or the Strichartz estimates.

Let us now consider randomization of general functions of 
negative regularity.
Given $s  < 0$, fix $(u_0, u_1) \in \H^s(\T^d) \setminus \H^0(\T^d)$ 
and let $(u_0^\o, u_1^\o)$ be its randomization
defined in \eqref{rand2}.
We then define the random linear solution $z$ by 
\begin{align}
\begin{split}
z(t) & =  \cos(t \jb{\nb}) u_0^\o + \frac{\sin(t\jb{\nb})}{\jb{\nb}} u_1^\o\\
&  = \sum_{n \in \Z^d} 
\bigg(\cos (t\jb{n}) g_n(\o) a_n + \frac{\sin(t\jb{n})}{\jb{n}} h_n(\o) b_n\bigg)e^{in \cdot x}.
\end{split}
\label{Z1}
\end{align}

\noi
Given $N \in \N$, 
we  set $z_N = \P_N z$ 
and  
\begin{align*}
\al_N (t) \stackrel{\text{def}}{=}  \E\big[ z_{N}^2(t, x) \big]
= \sum_{|n| \le N} 
\bigg(\cos^2 (t\jb{n}) |a_n|^2 + \frac{\sin^2(t\jb{n})}{\jb{n}^2} | b_n|^2\bigg)
\end{align*}

\noi
which is divergent in general as $N \to \infty$ since $(u_0, u_1) \notin \H^0(\T^3)$.  
Note that $\al_N$ depends on time in general.
For example, if we take $u_1 = \jb{\nb} u_0$, then
\begin{align*}
\al_N 
=  \sum_{|n| \le N}  |a_n|^2 \too \infty, 
\end{align*}		

\noi
as $N\to \infty$.
As in \eqref{Wick1}, 
given $j \in \N$, 
we 
define the truncated Wick power:
\begin{equation}
:\!z_{N}^j (t, x)\!:\, \,  =  H_j (z_{N}(t, x); \al_N(t)).
\label{Wick3}
\end{equation}

\noi
The following result shows that
for general  $(u_0, u_1) \in \H^s(\T^d) \setminus \H^0(\T^d)$, 
the truncated Wick powers $:\!z_{N}^j \!:$
do not converge even as  distributions
and that we need to impose additional  Fourier-Lebesgue regularity.
Given $s \in \R$ and $1 \le p \le \infty$, 
we define  the Fourier-Lebesgue space $\F L^{s, p}(\T^d)$
via the norm:
\begin{align*}
\| f\|_{\F L^{s, p}} = \| \jb{n}^s \ft f(n)\|_{\l^p(\Z^d)}
\end{align*}

\noi
and we set $\vec {\F L}^{s, p}(\T^d)
=  {\F L}^{s, p}(\T^d)\times  {\F L}^{s-1, p}(\T^d)$.
We state our main result.

\begin{theorem}\label{THM:1}
\textup{(i)} Given $s  < 0$, fix $(u_0, u_1) \in \H^s(\T^d) \setminus \H^0(\T^d)$ 
and let $(u_0^\o, u_1^\o)$ be its randomization
defined in \eqref{rand2}.
Given 
an integer $j \ge 2$, 
 let $:\!z_{N}^j \!:$ be the truncated Wick power defined in~\eqref{Wick3}.
Let $\s \le j s$ and $p > 2$.
Suppose that 
one of the following conditions holds\textup{:}
\begin{itemize}
\item[(i.a)]
$ \s \ge - \frac d2$ 
and $2 <  p  < p_{d, j, \s} \stackrel{\textup{def}}{=}\frac{2dj}{dj+ 2\s}$, or 

\smallskip

\item[(i.b)]
$\s \le -\frac d2$ and $2 <  p \le  \frac{2j}{j-1}$ $(=  p_{d, j, -\frac  d2})$.
\end{itemize}

\noi
If, in addition, we have 
$(u_0, u_1) \in \vec{\F L}^{0, p}(\T^d)$, 
then, 
given any finite $r \ge 1$ and $T > 0$, 
the sequence $\{:\!z_{N}^j \!:\}_{N \in \N}$
converges to a limit, denoted by 
$:\!z^j \!:$, in $C([0, T]; W^{\s, r}(\T^d))$
almost surely, as $N \to \infty$.

\medskip

\noi
\textup{(ii)}
Given an integer $j \ge 2$, 
there exists  $(u_0, u_1) \in \big(\bigcap_{s < 0} \H^s(\T^d)\big) \setminus
\vec{\F L}^{0, \frac {2j}{j-1}}(\T^d)$
such that the following statements hold 
for any $\s \in  \R$, almost surely\textup{:}

\smallskip

\begin{itemize}
\item[(ii.a)]
 Given any $t \in \R$ and $T > 0$,  the truncated Wick power
 $:\!z_{N}^j (t)\!:$ defined in~\eqref{Wick3}
 does not converge to any limit
in  $C([0, T]; H^\s(\T^d))$
or  
  $C([0, T]; \mathcal D'(\T^d))$.

\smallskip

\item[(ii.b)] The sequence
 $\mathcal I (\,:\!z_{N}^j \!:\,)$ 
 does not converge to any limit in $C([0, T]; H^\s(\T^d))$
 or   $C([0, T]; \mathcal D'(\T^d))$, 
 where 
$\mathcal I$ denotes the wave Duhamel integral operator defined by 
\begin{align}
 \mathcal{I} (F) (t) = \int_0^t \frac{\sin ((t - t') \jb{\nb})}{\jb{\nb}} F(t') dt'.
\label{duhamel}
\end{align}

\end{itemize}

\noi
In particular, 
the Wick renormalized NLW 
\begin{align}
\dt^2 v +(1-\Delta) v+\NN_k(v+z) = 0, 
\label{NLW2}
\end{align}

\noi
where $\NN_k$ is as in \eqref{Wick2}, 
is probabilistically ill-posed
with respect to randomization of general functions
in the sense that the  standard solution
theory such as the first order expansion or its variant based on a higher order expansion 
fails.

\end{theorem}

When $j = 2$, Theorem \ref{THM:1}\,(ii.a) and (ii.b) hold
for the  pair $(u_0, u_1) = (u_0, \jb{\nb}u_0)$
with {\it any} 
$u_0 \in \mathcal D'(\T^d) \setminus \F L^{0, 4}(\T^d)$.
See Remark \ref{REM:div1}.

Since $z \in H^s(\T^d)$, 
we expect that $:\!z^j (t)\!:$ has  at best regularity $j s$, 
and thus  the condition  $\s \le js \, (<0)$ in (i.a) is a natural one to impose.
Note that $\F L^{0, p_{d, j, \s}}(\T^d)$ scales like $H^\frac \s j(\T^d)$
where   $p_{d, j, \s}$ is as in  Theorem \ref{THM:1}\,(i),
and by 
 H\"older's inequality, we have
\begin{align*}
\| u_0 \|_{H^\frac{\s}{j}} \les \| u_0 \|_{\F L^{0, p}}
\end{align*}

\noi
for any $1 \le p < p_{d, j, \s}$ for $-\frac d2 \le \s \le s$ (here, $\s, s$ are negative).

Theorem \ref{THM:1}\,(ii) is of particular interest
since it shows existence of initial data $(u_0, u_1)$
which barely misses being in  $\H^0(\T^d)$
but for which the standard probabilistic well-posedness theory fails.
In particular, Theorem \ref{THM:1}\,(ii) shows that 
the claim in  \cite[Remark 1.2]{OPTz} is not correct.
In the context of 
the cubic NLW on $\T^3$ studied in \cite{OPTz}, 
Theorem \ref{THM:1}\,(i) provides
the following probabilistic local well-posedness result.

\begin{corollary}\label{COR:2}
Let $d = 3$ and  $k = 3$.
Given  $ -\frac 16 < s < 0$, 
let
$(u_0, u_1) \in \big(\H^s(\T^d) 
\cap \vec{\F L}^{0, p}(\T^d)\big)
\setminus \H^0(\T^d)$ 
for some $2 <  p < p_{3, 3, 3s}$, 
where $p_{d, j, \s}$ is as in (i.a) in Theorem \ref{THM:1}\,(i), 
and let $(u_0^\o, u_1^\o)$ be its randomization defined in \eqref{rand2}.
Then, almost surely, 
there exist  $T_\o > 0$
and a unique solution $v$ to 
the Wick renormalized NLW \eqref{NLW2} on the time interval $[0, T_\o]$.
\end{corollary}

Theorem \ref{THM:1}\,(i) yields
that $ :\!z^j \!: $ almost surely belongs to  
$ C([0, T]; W^{js, r}(\T^d))$ for any finite $r \ge 1$ and $T > 0$, $j = 1, 2,  3$.
In particular, we have 
$\mathcal I (:\!z^3 \!:)
\in C([0, T]; W^{3s + 1, r}(\T^3))$ almost surely, where $3s + 1 > \frac 12$
(i.e.~subcritical regularity for the 3-$d$ cubic NLW).
Then, Corollary \ref{COR:2} 
follows from a standard contraction argument 
with the Strichartz estimates
and the product estimates in 
\cite[Lemma 3.4]{GKO}.
We omit details.

\begin{remark}\label{REM:Euc}\rm
(i) There is a gap between the sufficient conditions
given in Theorem \ref{THM:1}\,(i)
and the necessary condition
given in Theorem \ref{THM:1}\,(ii)
for convergence of the truncated Wick powers  $:\!z_{N}^j \!:$.
Moreover, Theorem \ref{THM:1}\,(i)
only discusses the construction of the Wick powers $:\!z^j \!:$.
In order to better understand probabilistic well-posedness
with general randomized initial data of negative regularity, 
it is  important to study 
multilinear smoothing 
under the Duhamel integral operator $\mathcal I$
(as in \cite{Bring, OWZ, BDNY}; see also \cite{BO96})
and more complex stochastic objects which appear in 
higher order expansions (see \cite{GKO2, OPTz, Bring})
in the case of general random initial data.
For conciseness of the presentation, we do not pursue these issues
in this paper.

\smallskip

\noi
(ii)
In \cite{CO}, Colliander and the first author studied
probabilistic well-posedness of the cubic NLS on $\T$
with the random initial data $\varphi_0^\o$ in \eqref{rand3}.
A quick investigation suggests that, 
in order to consider randomization of a general function
of negative regularity
as initial data, additional Fourier-Lebesgue is needed.
Hence, it is worthwhile to investigate if 
 an analogue of Theorem \ref{THM:1} holds for NLS on $\T^d$.


\smallskip

\noi
(iii) Over the last  decade,  there have also been
intense research activities
(see, for example, \cite{LM, BOP1, BOP2,  Poc,  OOP, BOP3, DLM})
on probabilistic well-posedness of nonlinear dispersive PDEs
on the Euclidean space $\R^d$, 
where random initial data is often given 
by the Wiener randomization of a given function on $\R^d$, analogous to \eqref{rand2};
see \cite{BOP1, BOP2} for the Wiener randomization procedure.
So far, there is no probabilistic well-posedness result
with respect to the Wiener randomization of a general function
of negative Sobolev regularity (without an extra assumption such as radiality),\footnote{There is a recent paper \cite{HY}, 
where the authors claim almost sure local well-posedness of the quintic NLS on $\R$ 
with respect to the Wiener randomization of a function below $L^2(\R)$, 
but unfortunately,  their proof of this result is not correct.
Their argument is based on the probabilistic bilinear Strichartz estimate
(\cite[Proposition~2.8]{HY}), where one of the functions is assumed to be {\it deterministic}, 
and it is obviously false to apply such an estimate 
in a Picard iteration argument, starting with a random linear solution.}
and thus it is of interest 
to study if additional Fourier-Lebesgue regularity
is needed for probabilistic well-posedness 
 for NLW or NLS on $\R^d$ 
with respect to the Wiener randomization of a general function
of negative Sobolev regularity.

\smallskip

\noi
(iv)  A triviality result in the study of singular random PDEs says that
if one starts with regularized random initial data (or regularized stochastic forcing)
but without renormalization on a nonlinearity, 
then as one removes regularization, the corresponding solutions converge
to a trivial function or a linear solution.
See \cite{OPTz, OOR} for such triviality results on random NLW.
See also \cite{GO, OW, Ch1, Ch2} for triviality results for other dispersive PDEs
(even in the deterministic setting).
It is an intriguing question to see if the triviality results
in \cite{OPTz, OOR} extend to the case of general random initial data
in \eqref{rand2}.

\end{remark}

\subsection{On general stochastic forcing}
\label{SUBSEC:sto}\rm

Let us consider SNLW \eqref{NLW0a}.
For simplicity, we consider the zero initial data
and assume that $\Phi$ is a Fourier multiplier operator; namely, 
$\Phi(f) = \phi*f$  for some distribution $\phi$ on $\T^d$.
The basic stochastic object in the study of \eqref{NLW0a} is the stochastic convolution $\Psi$
defined by 
\begin{align}
\Psi(t) 
& = \int_0^t 
 \frac{\sin ((t - t') \jb{\nb})}{\jb{\nb}} \Phi \xi ( dt')
=  \sum_{n \in \Z^d}  \frac{\ft \phi_n I_n(t) }{\jb{n}} e^{in\cdot x}, 
\label{psi1}
\end{align}

\noi
where $I_n(t)$ is the Wiener integral given by 
\begin{align}
 I_n(t) =  \int_0^t \sin ((t - t') \jb{n}) d\be_n(t').
\label{psi1a}
\end{align}
Here, 
$\{\be_n \}_{n \in \Z^d}$ 
is defined by 
$\be_n(t) = \jb{\xi, \ind_{[0, t]} \cdot e^{in \cdot x}}_{t, x}$, 
where  $\xi$ is the space-time white noise 
and $\jb{\cdot, \cdot}_{t, x}$ denotes 
the duality pairing on $\R_+\times \T^d$.
Namely,  $\{ \be_n \}_{n \in \Z^d}$ is a family of mutually independent complex-valued
Brownian motions conditioned that $\be_{-n} = \cj{\be_n}$, $n \in \Z^d$. 
As a consequence, we see that $\{I_n(t)\}_{n \in \Z^d}$ is a family of 
mutually independent mean-zero complex-valued Gaussian random variables
 with variance $\sim t$, 
conditioned that $I_{-n}(t) = \cj{I_n(t)}$, $n \in \Z^d$.

When $\ft \phi_n = \jb{n}^{-\al}$, 
the regularity properties of $\Psi$ in \eqref{psi1}
is essentially the same as the random linear solution $Z$ defined in 
\eqref{rand4} with the random initial data $(\varphi_0^\o, \varphi_1^\o)$ in \eqref{rand3}, 
and thus the Wick power $:\!\Psi^k\!:$ can be defined
via a limiting procedure, 
just as in $:\!Z^k\!:$, 
provided that  $\al$ is sufficiently close to $ \frac d2 - 1$.

Before we move onto the general case, 
we recall the notion of 
$\g$-radonifying operators.
We say that a Fourier multiplier operator $\Phi$ is a $\g$-radonifying operator
from $L^2(\T^d)$ to $\F L^{s, p}(\T^d)$
if $\phi \in \F L^{s, p}(\T^d)$, where $\phi$ is the convolution kernel 
of  $\Phi$.
See \cite[(1.11) and Appendix A]{FOW}
for a further discussion and references.
Then,  a slight modification of the proof of Theorem \ref{THM:1} yields
the following result.
Recall that we assume that $\Phi$ is a Fourier multiplier operator.

\begin{theorem}\label{THM:3}
Given $s  < 0$, 
let $\Phi$ be a Hilbert-Schmidt operator
from $L^2(\T^d)$ into $H^{s-1}(\T^d)$
and $\Psi$ be as in \eqref{psi1}.
Given 
an integer $j \ge 2$, 
 let $:\!\Psi_{N}^j \!: $ be the truncated Wick power defined
 as in \eqref{Wick3} \textup{(}with $\Psi_N = \P_N \Psi$
 in place of $z_N$\textup{)}.
Let $\s \le j s$ and $p > 2$.
Suppose that 
one of the following conditions holds\textup{:}
\begin{itemize}
\item[(i.a)]
$ \s \ge - \frac d2$ 
and $2  <  p  < p_{d, j, \s} \stackrel{\textup{def}}{=}\frac{2dj}{dj+ 2\s}$, or 

\smallskip

\item[(i.b)]
$\s \le -\frac d2$ and $2 <  p \le  \frac{2j}{j-1}$ $(=  p_{d, j, -\frac  d2})$.
\end{itemize}

\noi
If, in addition, 
$\Phi$ is a $\g$-radonifying operator
from $L^2(\T^d)$ to $\F L^{-1, p}(\T^d)$, 
then, 
given any finite $r \ge 1$ and $T > 0$, 
the sequence $\{:\!\Psi_{N}^j \!:\}_{N \in \N}$
converges to a limit, denoted by 
$:\!\Psi^j \!:$, in $C([0, T]; W^{\s, r}(\T^d))$
almost surely, as $N \to \infty$.

\medskip

\noi
\textup{(ii)}
Given an integer $j \ge 2$, 
there exists 
a Hilbert-Schmidt operator $\Phi$
from $L^2(\T^d)$ into $H^{s-1}(\T^d)$
for any $s < 0$, 
which is not a $\g$-radonifying operator
from $L^2(\T^d)$ into
$\F L^{-1, \frac {2j}{j-1}}(\T^d)$
such that the following statements hold 
for any $\s \in  \R$, almost surely\textup{:}

\smallskip

\begin{itemize}
\item[(ii.a)]
 Given any $t \in \R$ and $T > 0$,  the truncated Wick power
 $:\!\Psi_{N}^j (t)\!:$ 
 does not converge to any limit
in  $C([0, T]; H^\s(\T^d))$
or   $C([0, T]; \mathcal D'(\T^d))$.

\smallskip

\item[(ii.b)] The sequence
 $\mathcal I (\,:\!\Psi_{N}^j \!:\,)$ 
 does not converge to any limit in $C([0, T]; H^\s(\T^d))$
 or   $C([0, T]; \mathcal D'(\T^d))$, 
 where 
$\mathcal I$ is as in \eqref{duhamel}.

\end{itemize}

\noi
In particular, 
the Wick renormalized SNLW 
\begin{align*}
\dt^2 v +(1-\Delta) v+\NN_k(v+\Psi) = 0, 
\end{align*}

\noi
where $\NN_k$ is as in \eqref{Wick2}, 
is  ill-posed
in the sense that the  standard solution
theory such as the first order expansion or its variant based on a higher order expansion 
fails.

\end{theorem}

By noting that  $\ft \phi_n$ in \eqref{psi1}
essentially plays a role of $b_n$ in \eqref{rand2}, 
Theorem \ref{THM:3} follows
from a straightforward modification of the proof of Theorem \ref{THM:1}
and thus we omit details.
See \cite{FOW} for an example, where local well-posedness
of the stochastic cubic NLS on $\T$
was established for singular noises by imposing 
an appropriate 
Fourier-Lebesgue
$\g$-radonifying regularity.

\begin{remark}\rm
(i) 
When $j = 2$, Theorem \ref{THM:3}\,(ii.a) and (ii.b) hold
for {\it any} Fourier multiplier operator $\Phi$ on $L^2(\T^d)$
which is not $\g$-radonifying
from $L^2(\T^d)$ into $\F L^{-1, 4}(\T^d)$.

\smallskip

\noi
(ii)
Consider the following stochastic nonlinear heat equation (SNLH):
\begin{equation}
\partial_t u +(1-\Delta) u+u^k = \Phi \xi.
\label{NLH}
\end{equation}

\noi
Let $\Psi_\text{heat}$ be the associated stochastic convolution:
\begin{align}
\Psi_\text{heat}(t) 
& = \int_0^t 
e^{ (t - t')(\Dl - 1)} \Phi \xi ( dt')
=  \sum_{n \in \Z^d}  \ft \phi_n J_n(t) e^{in\cdot x}.
\label{psi3}
\end{align}

\noi
Here,  $J_n(t)$ is the Wiener integral given by 
\[ J_n(t) =  \int_0^t 
e^{ - (t - t')\jb{n}^2} d\be_n(t'), \] 

\noi
where $\be_n$ is as in \eqref{psi1a}.
It is easy to see that  $\{J_n(t)\}_{n \in \Z^d}$ is a family of 
mutually independent mean-zero complex-valued Gaussian random variables
 with variance $\sim \jb{n}^{-1}$, 
conditioned that $J_{-n}(t) = \cj{J_n(t)}$, $n \in \Z^d$,
and hence that 
an analogue of 
 Theorem \ref{THM:3}
also holds for $\Psi_{\text{heat}}$ in \eqref{psi3}
and SNLH \eqref{NLH}.
When $j = 2$, Part (i) of this remark also holds in this case.

\end{remark}

\section{Proof of Theorem \ref{THM:1}}

\subsection{Construction of Wick powers}
In this subsection, we present the proof of Theorem~\ref{THM:1}\,(i).
We first recall the following orthogonality relation for the Hermite polynomials
(\cite[Lemma~1.1.1]{Nua}).

\begin{lemma}\label{LEM:Wick}
Let $f$ and $g$ be mean-zero jointly Gaussian random variables with variances $\s_f$
and $\s_g$.
Then, we have 
\begin{align*}
\E\big[ H_k(f; \s_f) H_m(g; \s_g)\big] = \dl_{km} k! \big\{\E[ f g] \big\}^k.
\end{align*}
\end{lemma}

Let $(u_0, u_1) \notin \H^0(\T^d)$ be as in \eqref{rand1}.
Given $n \in \Z^d$ and $t \in \R$, define $\g_n(t)$ by 
\begin{align}
\g_n (t)
= \cos^2 (t\jb{n}) |a_n|^2 + \frac{\sin^2(t\jb{n})}{\jb{n}^2} | b_n|^2.
\label{Z4}
\end{align}

\noi
Then, 
from  Lemma \ref{LEM:Wick} with \eqref{Z1}, we have
\begin{align}
\begin{split}
\E\Big[|\F_x(:\!z_{N}^j (t)\!:)(n)|^2\Big]
& = \int_{\T^d_x\times \T^d_y}
\E\Big[:\!z_{N}^j (t, x)\!:   \cj{:\!z_{N}^j (t, y)\!:}\Big]
e^{ - in \cdot (x-y)} dx dy \\
& = j! \int_{\T^d_x\times \T^d_y}
\bigg(\prod_{\l = 1}^j \sum_{\substack{n_\l \in \Z^d\\|n_\l| \le N}}\g_{n_\l} (t) e^{i n_\l \cdot (x-y)}\bigg)
e^{- in \cdot (x-y)} dx dy \\
& = j! \sum_{\substack{n_\l \in \Z^d\\n = n_1 + \cdots + n_j\\|n_\l|\le N}} \prod_{\l = 1}^j\g_{n_\l} (t). 
\end{split}
\label{Z5}
\end{align}

\noi
Thus, from \eqref{Z5} and \eqref{Z4}, we have, for any $\s \in \R$, 
\begin{align}
\begin{split}
\E \Big[\| & :\!z_{N}^j (t)\!:\|_{H^\s}^2\Big]
 = \sum_{n \in \Z^d}
\jb{n}^{2\s} \E\Big[|\F_x(:\!z_{N}^j (t)\!:)(n)|^2\Big]\\
& =  j! 
\sum_{\substack{|n_\l|\le N\\ \l = 1, \dots, j}} 
\jb{n_1 + \cdots +n_j}^{2\s}
\prod_{\l = 1}^j \bigg( \cos^2 (t\jb{n_\l}) |a_{n_\l}|^2 + \frac{\sin^2(t\jb{n_\l})}{\jb{n_\l}^2} | b_{n_\l}|^2\bigg).
\end{split}
\label{Z6}
\end{align}

\noi
When $\s \ge 0$, we have
 $\jb{n_1 + \cdots +n_j}^{2\s}  \les \prod_{\l = 1}^j \jb{n_\l}^{2\s}$.
 However, 
when $\s  < 0$, such an inequality is false, 
which allows us to show that the right-hand side is divergent
for a suitable choice of $(u_0, u_1)$.

Before presenting the proof of Theorem \ref{THM:1}\,(i), 
let us first consider the random initial data $(\varphi_0^\o, \varphi_1^\o)$ in~\eqref{rand3}.
In  the construction
of the Wick powers  of the truncated random linear solution $Z_N = \P_N Z$, 
where $Z$ is as in \eqref{rand4}, 
the right-hand side of \eqref{Z6} (dropping $j!$) is given by 
the following  iterated discrete convolutions:
\begin{align}
\begin{split}
\sum_{\substack{|n_\l|\le N\\\l = 1, \dots, j}} 
\jb{n_1 + \cdots +n_j}^{2\s}
\prod_{\l = 1}^j \frac{1}{\jb{n_\l}^{2(1+\al)}}.
\end{split}
\label{Z8}
\end{align}

\noi
By  iteratively carrying out summations
(via Lemma~3.4 in \cite{OOT1}), 
we see that \eqref{Z8} is uniformly bounded in $N \in \N$ for $\s <j s(d,\al) \le 0$, 
where $s(d, \al)$ is as in \eqref{al0}, 
provided that $s(d, \al)$ is sufficiently close to $0$.
By viewing 
$(\varphi_0^\o, \varphi_1^\o)$ in~\eqref{rand3}
as the randomization of a pair $(\varphi_0, \varphi_1)$ whose Fourier
coefficients are given by $\jb{n}^{-1-\al}$ and $\jb{n}^{-\al}$, respectively, 
we indeed used the Fourier-Lebesgue regularity 
of  $(\varphi_0, \varphi_1)$ in bounding \eqref{Z8}.

\begin{remark}\rm

When $2\s < -d$, we can bound \eqref{Z8} by 
\begin{align}
\sup_{n \in \Z^d} \sum_{n = n_1 + \cdots + n_j}
\prod_{\l = 1}^j \frac{1}{\jb{n_\l}^{2(1+\al)}}, 
\label{Z8a}
\end{align}

\noi
which 
yields a necessary condition 
$ 2j (1+\al) > (j-1)d$
for summability of \eqref{Z8a}.
For example,  when $d = 3$, $j = 3$,  and $\al = 0$, 
this condition is violated
which is consistent with the non-existence
of the cubic Wick power of the Gaussian free field.

\end{remark}

Let us go back to the case of general randomized initial data \eqref{rand2}
and present the proof of Theorem \ref{THM:1}\,(i).
We first consider the case $ - \frac d2 \le \s \le js$.
Given small $\eps_0 > 0$, set finite $q \ge 1$ by 
\begin{align}
\frac{1}q = -\frac {2\s}d - \eps_0
\quad \text{such that}
\quad 2\s q< -d.
\label{Z9a}
\end{align}

\noi
Note that we used the condition 
$- \frac d2 \le \s < 0$ to guarantee that $q$ in \eqref{Z9a}
satisfies $q \ge 1$.
Then, 
from  \eqref{Z6} and H\"older's inequality with \eqref{Z9a}, we have 
\begin{align}
\begin{split}
\E \Big[\|  :\!z_{N}^j (t)\!:\|_{H^\s}^2\Big]
& \les\bigg\|\sum_{n = n_1 + \cdots + n_j}
\prod_{\l = 1}^j \bigg( |a_{n_\l}|^2 + \frac{|b_{n_\l}|^2 }{\jb{n_\l}^2} |\bigg)\bigg\|_{\l^{q'}_n}.
\end{split}
\label{Z9b}
\end{align}

\noi
In the following, we iteratively apply Young's inequality.
Let $p_0 = q'$
and 
\begin{align}
 \frac{1}{p_\l} + 1 = \frac{1}{p/2}  + \frac 1 {p_{\l+1}}, \quad \l = 0, 1, \dots, j-2
 \label{Z11}
\end{align}

\noi
with $p_{j-1} = \frac p 2>  1$.
Then, from \eqref{Z11} and \eqref{Z9a}, we have 
\begin{align}
\frac 1p = \frac{1}{2}-\frac{1}{2jq}
= \frac{dj+2\s}{2dj} +\eps_1, 
\label{Z14}
\end{align}

\noi
where $\eps_1 = \frac 1{2j}\eps_0$.
Let $c_n =  |a_{n}| + \frac{|b_n|}{\jb{n}}$
such that $\|c_n\|_{\l^p_n} \sim \| (u_0, u_1) \|_{\vec{\F L}^{0, p}}$. 
Then, 
by iteratively applying Young's inequality to \eqref{Z9b}, we obtain
\begin{align}
\begin{split}
\E \Big[\|  :\!z_{N}^j (t)\!:\|_{H^\s}^2\Big]
& \les 
\| (u_0, u_1) \|_{\vec{\F L}^{0, p}}^2
\bigg \| \sum_{m_{j-1} = n_1 + \cdots + n_{j-1}}
\prod_{\l = 1}^{j-1} \bigg( |a_{n_\l}|^2 + \frac{|b_{n_\l}|^2 }{\jb{n_\l}^2} |\bigg)
\bigg\|_{\l^{p_1}_{m_{j-1}}}\\
& \les \| (u_0, u_1) \|_{\vec{\F L}^{0, p}}^4
\bigg \| \sum_{m_{j-2} = n_1 + \cdots + n_{j-2}}
\prod_{\l = 1}^{j-2} \bigg( |a_{n_\l}|^2 + \frac{|b_{n_\l}|^2 }{\jb{n_\l}^2} |\bigg)
\bigg\|_{\l^{p_2}_{m_{j-2}}}\\
& \les \cdots \les
\| (u_0, u_1) \|_{\vec{\F L}^{0, p}}^{2(j-1)}
\bigg\| |a_{n_1}|^2 + \frac{|b_{n_1}|^2 }{\jb{n_1}^2} 
\bigg\|_{\l^{p_{j-1}}_{n_{1}}} \\
& \les \| (u_0, u_1) \|_{\vec{\F L}^{0, p}}^{2j} < \infty, 
\end{split}
\label{Z14a}
\end{align}

\noi
uniformly in $N \in \N$, 
provided that 
$(u_0, u_1) \in \vec {\F L}^{0,p}(\T^d)$ 
for some $2 <  p < \frac{2dj}{dj+2\s}$.

Next, we consider the case $\s < - \frac{d}{2}$.
In this case, from \eqref{Z9b}, we have
\begin{align*}
\E \Big[\|  :\!z_{N}^j (t)\!:\|_{H^\s}^2\Big]
& \les\bigg\|\sum_{n = n_1 + \cdots + n_j}
\prod_{\l = 1}^j \bigg( |a_{n_\l}|^2 + \frac{|b_{n_\l}|^2 }{\jb{n_\l}^2} |\bigg)\bigg\|_{\l^{\infty}_n}.
\end{align*}

\noi
With $p_0 = \infty$, we recursively define $p_\l$ as in \eqref{Z11}.
Then, 
from   \eqref{Z14} with $q = 1$, 
we have $p = \frac {2j}{j-1}$.
Then, 
by iteratively applying Young's inequality as in \eqref{Z14a}, we obtain
\begin{align}
\begin{split}
\E \Big[\|  :\!z_{N}^j (t)\!:\|_{H^\s}^2\Big]
& \les \| (u_0, u_1) \|_{\vec{\F L}^{0, \frac{2j}{j-1}}}^{2j} < \infty, 
\end{split}
\label{Z15}
\end{align}

\noi
uniformly in $N \in \N$.

Once we have the uniform bound \eqref{Z14a} or \eqref{Z15}, 
almost sure convergence of $:\!z_{N}^j\!:$
in $C([0, T]; W^{\s, r}(\T^d))$ for any finite $r \ge 1$
claimed in 
Theorem \ref{THM:1}\,(i)
follows 
from a standard argument,  involving 
the Wiener chaos estimate (see \cite[Proposition 2.4]{TTz}) 
and Kolmogorov's continuity criterion-type argument, 
and hence we omit details.
See, for example, \cite{GKO, GKO2, OPTz, OOTz}.

\subsection{Counterexample}
In this subsection, we present the proof of Theorem \ref{THM:1}\,(ii).
We define $u_0$ on $\T^d$
whose Fourier coefficient at 
the frequency 
 $n=(n^{(1)}, \dots, n^{(d)}) \in \Z^d$
is given by 
\begin{equation}
a_n
= \wt a_{n^{(1)}} \dots \wt a_{n^{(d)}},
\label{anA1}
\end{equation}
where
$\wt a_{n^{(i)}}$,  $i=1, \dots, d$, is defined by 
\begin{equation}
\wt a_{n^{(i)}}
= \begin{cases}
m^{-\frac{j-1}{2j}}, & \text{if there is $m \in \N$ such that } |n^{(i)}|=2^m, \\
0, & \text{otherwise.}
\end{cases}
\label{anA2}
\end{equation}

\noi
We set 
$u_1= \jb{\nb}u_0$.
Then, we have 
\begin{align*}
\| (u_0, u_1) \|_{\H^s}^2
&\sim
\sum_{n= (n^{(1)}, \dots, n^{(d)})}
\jb{n}^{2s} |\wt a_{n^{(1)}} \dots \wt a_{n^{(d)}}|^2
\les
\prod_{i=1}^d
\sum_{n^{(i)}=1}^\infty
\jb{n^{(i)}}^{\frac{2s}d} |\wt a_{n^{(i)}}|^2
\\
&\les
\Big(
\sum_{m=1}^\infty
2^{\frac{2s}d m} m^{-\frac {j-1}{j}}
\Big)^d
< \infty
\end{align*}

\noi
for any $s<0$.
Moreover, we have
\[
\| (u_0, u_1)  \|_{\vec {\F L}^{0, \frac{2j}{j-1}}}
\ges
\prod_{i=1}^d
\sum_{n^{(i)}=1}^\infty
|\wt a_{n^{(i)}}|^{\frac{2j}{j-1}}
=
\Big(
\sum_{m=1}^\infty m^{-1}
\Big)^d
= \infty.
\]

\noi
Hence, we conclude that  $(u_0, u_1) \in \big(\bigcap_{s < 0} \H^s(\T^d)\big) \setminus
\vec{\F L}^{0, \frac {2j}{j-1}}(\T^d)$.

Let $t \in \R$.
From \eqref{Z6}, we have 
\begin{align}
\begin{split}
\E \Big[\|  :\!z_{N}^j (t)\!:\|_{H^\s}^2\Big]
&  \sim
\sum_{\substack{|n_\l|\le N\\\l = 1, \dots, j}} 
\jb{n_1+ \dots +n_j}^{2\s}
\prod_{\l = 1}^j  |a_{n_\l}|^2 \\
&  \ge 
\sum_{|n_1|, \dots, |n_{j-1}| \le N}
\ind_{|n_1+\dots+n_{j-1}| \le N}
\Big(
\prod_{\l= 1}^{j-1}  |a_{n_\l}|^2 
\Big)
|a_{n_1+\dots+n_{j-1}}|^2, 
\end{split}
\label{Z18}
\end{align}

\noi
where the second step follows from 
considering  the contribution only for $n_1 + \cdots +n_j = 0$.

For $i=1, \dots, d$ and $\l=1, \dots, j-1$, set
\begin{align}
\Nf_\l^{(i)} := n_1^{(i)}+ \dots + n_\l^{(i)}.
\label{Z18a}
\end{align}

\noi
Noting that
\[
\bigg\{ n = (n^{(1)}, \dots, n^{(d)}) \in \Z^d : 
\max_{i=1, \dots, d} |n^{(i)} | \le \frac N{\sqrt d} \bigg\}
\subset
\{ n \in \Z^d : |n| \le N \}, 
\]
it follows from 
 \eqref{anA1} and \eqref{Z18a} that 
\begin{equation}
\begin{aligned}
\text{RHS of } \eqref{Z18}
&\ge
\prod_{i=1}^d
\bigg(
\sum_{0 \le n_1^{(i)} \le \dots \le n_{j-1}^{(i)} \le \frac{N}{\sqrt d}}
\Big(
\prod_{\l = 1}^{j-1}  |\wt a_{n_\l^{(i)}}|^2 
\Big)
|\wt a_{\Nf_{j-1}^{(i)}}|^2
\bigg).
\end{aligned}
\label{Z18b}
\end{equation}

\noi
When $j=2$ (i.e.~$\frac{2j}{j-1} = 4$),
it follows from \eqref{anA2} and \eqref{Z18b} that
\begin{align}
\text{RHS of } \eqref{Z18b}
\ges
\prod_{i=1}^d
\Big(
\sum_{n_1^{(i)}=0}^{\frac{N}{ \sqrt d}} |\wt a_{n_1^{(i)}}|^4
\Big)
\ges
\Big(
\sum_{m=2}^{[\log_2 \frac N{ \sqrt d}]}
m^{-1}
\Big)^d
\sim
(\log  \log N)^d, 
\label{Z18c}
\end{align}

\noi
where $[x]$ denotes the integer part of $x \in \R$.

Now, we consider the case $j \ge 3$.
We first state a lemma whose proof is presented at the end of this section.

\begin{lemma}\label{LEM:Oka}
Let $j \ge 3$. Then, 
there exist $N(j) \in \N$ and 
small $c_j > 0$ such that 
\begin{equation}
\sum_{\substack{
n_{2}^{(i)}, \dots, n_{j-1}^{(i)} \in \N \\
4 \le n_2^{(i)} \le \dots \le n_{j-1}^{(i)} \le \frac{N}{\sqrt d}}}
\ind_{4 \le n_1^{(i)} \le c_j \frac{N}{\sqrt d}}
\cdot 
\Big(
\prod_{\l = 2}^{j-1}  |\wt a_{n_\l^{(i)}}|^2 
\Big)
|\wt a_{\Nf_{j-1}^{(i)}}|^2
\ges (\log_2 n_1^{(i)})^{-1+\frac{j-1}j}, 
\label{Oka1}
\end{equation}

\noi
uniformly in 
$N \ge N(j)$ and 
$i = 1, \dots, d$.

\end{lemma}

By Lemma \ref{LEM:Oka}, we obtain

\vspace{-7mm}

\begin{align}
\begin{split}
\text{RHS of } \eqref{Z18b}
& \ges
\prod_{i=1}^d
\bigg(
\sum_{n_1^{(i)}=4}^{c_j \frac{N}{\sqrt{d}}} |\wt a_{n_1^{(i)}}|^2 
(\log_2 n_1^{(i)})^{-1+\frac{j-1}j}
\bigg)\\
& \sim
\bigg(
\sum_{m=2}^{[ \log_2 c_j \frac{N}{\sqrt{d}}]}
m^{-1}
\bigg)^d
\sim
(\log  \log N)^d, 
\end{split}
\label{Oka1a}
\end{align}

\noi
where the last step hold
for any sufficiently large $N \gg 1$ (depending on $j$).
Therefore, from 
\eqref{Z18},
\eqref{Z18b}
 \eqref{Z18c}, and \eqref{Oka1a}, 
 we conclude that
\begin{align*}
\E \Big[\|  :\!z_{N}^j (t)\!:\|_{H^\s}^2\Big]
\ge \E \Big[|  \F_x(:\!z_{N}^j (t)\!:)(0)|^2\Big]
\ges
(\log  \log N)^d \too \infty, 
\end{align*}

\noi
as $N \to \infty$.
At this point, we can repeat the argument
in \cite[Subsection 4.4]{OOcomp}
with Kolmogorov's three series theorem and zero-one law
and conclude  Theorem \ref{THM:1}\,(ii.a).
We  omit details.
Since we only estimated the contribution from the zeroth frequency, 
we also conclude non-convergence in $C([0, T]; \mathcal D'(\T^d))$.
Noting that the wave Duhamel integral operator
does give any smoothing at the zeroth frequency,
we also obtain 
Theorem \ref{THM:1}\,(ii.b)

\medskip

We conclude this paper by presenting the proof of Lemma \ref{LEM:Oka}.

\begin{proof}[Proof of Lemma \ref{LEM:Oka}]

We restrict the sum on the left-hand side of \eqref{Oka1}
to 
\begin{align}
4\le n_{\l}^{(i)} \le  2^{-2^\frac{j}{j-2}} \frac{N}{j \sqrt d} =: M_{j-2}
\label{Oka2}
\end{align}

\noi
for $\l = 1, \dots, j-2$ (but not for $\l = j-1$).
With \eqref{Z18a}, 
this in particular implies
\begin{align}
2^{2^\frac{j}{j-2}} 
( \Nf^{(i)}_{j-2} +n_{j-2}^{(i)})
\le 
\frac{N}{ \sqrt d}
\le \Nf^{(i)}_{j-2} + \frac{N}{ \sqrt d}.
\label{Oka3}
\end{align}

\noi
Noting that $a b \ge a + b$
for $a, b \ge 2$, it follows from 
 \eqref{Oka3} that 
\begin{align}
\big(\log_2 ( \Nf^{(i)}_{j-2} +n_{j-2}^{(i)})\big)^{-1 + \frac 2j}
\ge 2 \bigg(\log_2 \Big(\Nf^{(i)}_{j-2} + \frac{N}{ \sqrt d}\Big)\bigg)^{-1 + \frac 2j}.
\label{Oka4}
\end{align}

\noi
Hence, from  \eqref{anA2} (in particular, 
$\wt a_{n}$ restricted to $n = 2^m$, $m \in \N$, is decreasing) 
and \eqref{Oka4}, we have
\begin{align*}
& \sum_{n_{j-1}^{(i)} = n_{j-2}^{(i)}}^{\frac{N}{ \sqrt d}}
|\wt a_{n_{j-1}^{(i)}}|^2 |\wt a_{\Nf^{(i)}_{j-2} + n_{j-1}^{(i)}}|^2
\ge
\sum_{n_{j-1}^{(i)} = n_{j-2}^{(i)}}^{\frac{N}{\sqrt d}}
|\wt a_{\Nf^{(i)}_{j-2} + n_{j-1}^{(i)}}|^4
\\
& \qquad 
\ge
\sum_{m=[ \log_2 (\Nf^{(i)}_{j-2} +n_{j-2}^{(i)})]+1}^{[\log_2 (\Nf^{(i)}_{j-2} + \frac{N} {\sqrt d})]}
m^{-2 + \frac 2j}
\ges
(\log_2 n_{j-2}^{(i)})^{-1+\frac 2j}.
\end{align*}

\noi
When $j = 3$, we stop  the calculation here.
When $j \ge 4$,
we further impose
\begin{align}
4\le n_{\l}^{(i)} \le  2^{-2^\frac{j}{j-3}} M_{j-2}
=: M_{j - 3}
\label{Oka5}
\end{align}

\noi
for $\l = 1, \dots, j-3$ (but not for $\l = j-2$), 
where $M_{j-2}$ is as in \eqref{Oka2}.
Then, 
we have
\begin{align}
&
\sum_{n_{j-2}^{(i)} =n_{j-3}^{(i)}}^{M_{j-2}}
|\wt a_{n_{j-2}^{(i)}}|^2 
(\log_2 n_{j-2}^{(i)})^{-1+\frac 2j}
\ges
\sum_{m= [\log_2 n_{j-3}^{(i)}]+1}^{[\log_2 M_{j-2}]}
m^{-\frac{j-1}j} m^{-1+\frac 2j}
\ges
(\log_2 n_{j-3}^{(i)})^{-1+\frac 3j}, 
\label{Oka5a}
\end{align}

\noi
where the last step follows from \eqref{Oka5}
(which implies $\big( \log_2 n_{j-3}^{(i)}\big)^{-1 + \frac3j}
\ge 2 \big( \log_2M_{j-2}\big)^{-1 + \frac3j}$).

In general, 
suppose  $j \ge k \ge 5$
and we assume that we have repeated the procedure above $k-3$ times.
In this case, 
we further impose a condition 
\begin{align}
4\le n_{\l}^{(i)} \le  2^{-2^\frac{j}{j-(k-1)}} M_{j-(k-2)}
=: M_{j - (k-1)}
\label{Oka6}
\end{align}

\noi
for $\l = 1, \dots, j-(k-1)$ (but not for $\l = j-(k-2)$).
The condition \eqref{Oka6} guarantees
 $\big( \log_2 n_{j-(k-1)}^{(i)}\big)^{-1 + \frac{k-1}j}
\ge 2 \big( \log_2M_{j-(k-2)}\big)^{-1 + \frac{k-1}j}$,
which allows us to 
repeat the computation 
as in \eqref{Oka5a} for the $(k-2)$nd step.
By iterating this procedure, 
we obtain \eqref{Oka1}
with 
$c_j
=  \prod_{k = 3}^j2^{-2^\frac{j}{j-(k-1)}} $
which follows from 
 \eqref{Oka2}, \eqref{Oka5}, and \eqref{Oka6}.
\end {proof}

\begin{remark}\label{REM:div1}\rm
Let $j = 2$ such that $\frac {2j}{j-1} = 4$.
Given any  $u_0 \in \mathcal D'(\T^d) \setminus L^2(\T^d)$, 
by setting   $u_1= \jb{\nb}u_0$, it follows from \eqref{Z6} 
and considering the contribution only from $n_1 + n_2 = 0$
that 
\begin{align*}
\E \Big[\|  :\!z_{N}^2 (t)\!:\|_{H^\s}^2\Big]
&  \ges
\sum_{|n_1|\le N}
 |a_{n_1}|^4
 \too  \infty, 
\end{align*}

\noi
as $N \to \infty$, 
for any $\s \in \R$
unless $u_0 \in \F L^{0, 4}(\T^d)$.
\end{remark}

\begin{ackno}\rm
T.O.~was supported by the European Research Council (grant no.~864138 ``SingStochDispDyn"). 
M.O.~was supported by JSPS KAKENHI  Grant number JP23K03182.
O.P.~was supported by the EPSRC New Investigator Award
(grant no. EP/S033157/1).
N.T.~was partially supported by the ANR project Smooth ANR-22-CE40-0017.
\end{ackno}

\vspace*{-2mm}

\end{document}